\documentclass[10pt]{amsart}

\usepackage[latin1]{inputenc}
\usepackage{cite}
\usepackage{amsmath}
\usepackage{amsfonts}
\usepackage{amssymb}
\usepackage{ytableau}
\usepackage[mathscr]{eucal}
\usepackage{diagrams}
\usepackage{xy,color,graphicx,tikz-cd}
\usepackage{hyperref}
\usepackage{enumitem}
\usepackage{microtype}
\usepackage{tabularx}
\xyoption{all}

\newcommand{\wis}[1]{{\text{\em \usefont{OT1}{cmtt}{m}{n} #1}}}

\newcommand{\NN}{\mathbb{N}}
\let\SS\relax
\newcommand{\SS}{\mathbb{S}}

\newcommand{\CC}{\mathbb{C}}
\newcommand{\ZZ}{\mathbb{Z}}

\newcommand{\op}{\mathrm{op}}

\DeclareMathOperator{\M}{M}
\DeclareMathOperator{\lcm}{lcm}

\DeclareMathOperator{\Hom}{Hom}

\DeclareMathOperator{\pgl}{PGL}
\DeclareMathOperator{\aut}{Aut}
\DeclareMathOperator{\en}{End}
\DeclareMathOperator{\z}{Z}
\newcommand{\rep}{\wis{rep}}
\newcommand{\trep}{\wis{trep}}
\newcommand{\spec}{\wis{spec}}

\newcommand{\comm}{\mathsf{Comm}}
\newcommand{\sets}{\mathsf{Set}}
\newcommand{\azu}{\mathsf{Azu}}

\newcommand{\sh}{\mathsf{Sh}}
\newcommand{\psh}{\mathsf{PSh}}
\newcommand{\alg}{\mathsf{Alg}}
\newcommand{\zalg}{\mathsf{ZAlg}}

\let\max\relax
\newcommand{\max}{\mathrm{max}}

\newcommand{\tr}{\mathrm{tr}}

\newtheorem{definition}{Definition}[section]
\newtheorem{proposition}[definition]{Proposition}

\newtheorem{theorem}[definition]{Theorem}
\newtheorem{corollary}[definition]{Corollary}
\newtheorem{lemma}[definition]{Lemma}
\newtheorem{example}[definition]{Example}
\newtheorem{remark}[definition]{Remark}

\tikzset{
b/.style={bend left=10},
bb/.style={bend left},
cl/.style={outer sep=-1pt},
}

\title{Azumaya toposes}
\author{Jens Hemelaer} 
\address{Department of Mathematics, University of Antwerp \\ 
Middelheimlaan 1, B-2020 Antwerp (Belgium) \\ {\tt jens.hemelaer@uantwerpen.be}}
\thanks{The author is a Ph.D.\ fellow of the Research Foundation -- Flanders (FWO)}
 
\begin{document}
\maketitle

\begin{abstract}
In \cite{azureps}, many different Grothendieck topologies were introduced on the category of Azumaya algebras. Here we give a classification in terms of sets of supernatural numbers. Then we discuss the associated categories of sheaves and their topos-theoretic points, which are related to UHF-algebras. The sheaf toposes that correspond to a single supernatural number have an alternative description, involving actions of the associated projective general linear group.
\end{abstract}

\section{Introduction}

All algebras will be defined over the complex numbers, and algebra morphisms will be $\CC$-linear.

Let $R$ be a finitely presented algebra, not necessarily commutative, and let $C$ be a commutative algebra. Then we define a \emph{family of $n$-dimensional representations} of $R$, parametrized by $C$, to be a choice of degree $n$ Azumaya algebra $A$ over $C$, together with an algebra morphism $R \longrightarrow A$. This is motivated by \cite{llb-stacks}, where it was shown that the $C$-points of the representation stack
\[
[~\rep_n R ~/~ \pgl_n~]
\]
are precisely given by these morphisms $R \longrightarrow A$ with $A$ some Azumaya algebra of degree $n$ over $C$ (where $\rep_n R$ is the scheme parametrizing $n$-dimensional representations of $R$). So the information regarding families of $n$-dimensional representations, for varying $n$, is contained in the functor
\[
\alg(R,-) : \azu \longrightarrow \sets \qquad A \mapsto \alg(R,A).
\]
Here $\azu$ is the category of finitely generated Azumaya algebras and center-preserving algebra morphisms. Notice the similarity with the functor of points approach to (commutative) algebraic geometry, where schemes $X$ are interpreted as functors $X : \comm \longrightarrow \sets$ with $\comm$ the category of (finitely generated) commutative rings. More precisely, schemes have an interpretation in the category of $J$-sheaves on $\comm^\op$, denoted $\sh(\comm^\op,J)$, where $J$ can be the Zariski topology, or for example the \'etale or flat topology.

In this paper we want to develop an analogous framework for the functors $\alg(R,-)$ as above. In \cite{azureps}, the first steps were taken in this direction. Many different Grothendieck topologies $J_K$ on $\azu^\op$ were constructed, and it was shown that for a lot of these topologies, $\alg(R,-)$ is a sheaf, regardless of the algebra $R$. These Grothendieck topologies $J_K$ are constructed from a Grothendieck topology $J$ on commutative rings and a Grothendieck topology $K$ on the \emph{big cell} $\mathtt{D}$, i.e.\ the category with as objects the nonzero natural numbers and a unique morphism $m \to n$ whenever $n \mid m$. The assumptions are that every $K$-cover admits a finite subcover, and that
\begin{enumerate}
\item $J$ is finer than the \'etale topology, or
\item $J$ is finer than the Zariski topology and $K$ is multiplicatively closed.
\end{enumerate}

In Section \ref{topologies-on-big-cell}, we give an explicit description of those Grothendieck topologies $K$ on the big cell such that every $K$-cover admits a finite subcover. Note that these always have enough points, by Deligne's completeness theorem. We will consider two topologies on the set $\SS$ of supernatural numbers: the \emph{pcfb-topology} and the \emph{patch topology}. The subsets $S \subseteq \SS$ that are closed for the pcfb-topology correspond bijectively to the Grothendieck topologies $K$ on the big cell that have enough points. Moreover, every $K$-cover admits a finite subcover if and only if the associated subset $S \subseteq \SS$ is closed for the (stronger) patch topology.

In Section \ref{trivializing-topologies}, we look at the topologies $J_K$ that are \emph{trivializing}, i.e.\ such that every Azumaya algebra is $J_K$-locally given by matrix algebras. The obvious examples are the topologies $J_K$ with $J$ finer than the \'etale topology. But we also determine the Grothendieck topologies $K$ such that $J_K$ is trivializing, for $J$ the Zariski topology.

In Section \ref{points}, we show that the trivializing topologies $J_K$ have enough points whenever $J$ has enough points. More explicitly, we show that the family
\[
\mathrm{P}(J,K) = \{ \M_s(D) \mid s \in S \text{ and }D\text{ a }J\text{-local commutative algebra} \}
\]
is a separating family of points for $J_K$, where each $\M_s(D)$ is a certain union of matrix algebras over $D$, similar to the UHF-algebra associated to $s$.

If $J_K$ is moreover coarser than the maximal flat topology, then $\alg(R,-)$ is a $J_K$-sheaf, for $R$ a finitely generated, not necessarily commutative algebra. In this case, we can associate a topos to the algebra $R$: the slice topos
\[
\sh(\azu^\op,J_K) \slash \alg(R,-).
\]
For these toposes, we find the family of points
\[
\mathrm{P}_R(J,K) = \{ R \to \M_s(D) \mid s \in S \text{ and }D\text{ a }J\text{-local commutative algebra} \},
\]
which is again separating whenever $J$ has enough points. So the topos-theoretic points we associate to $R$ correspond to certain representations parametrized by $J$-local commutative algebras.

In Section \ref{pgl-action} we construct a projective general linear group $\pgl_s$ over the complex numbers, for each supernatural number $s \in \SS$. The construction is analogous to the construction of UHF-algebras as unions of matrix algebras. Moreover, there is a natural action of $\pgl_s$ on the UHF-algebra $\M_s(\CC)$ and this action satisfies a finiteness condition that will be important later on. The $\pgl_s$-actions satisfying this condition will be called \emph{continuous}.

After studying $\pgl_s$, we take a closer look at the topologies $J_K$ with $K$ the Grothendieck topology on $\mathtt{D}$ induced by the singleton $\{s\}$ (we write $J_K=J_s$). We show that there is an equivalence of categories
\[
\sh(\azu^\op,J_s) \simeq \pgl_s - \sh(\comm^\op,J).
\]
The left hand side is the category of $J_s$-sheaves on $\azu^\op$. The right hand side is the category of $J$-sheaves on $\comm^\op$, equipped with a continuous $\pgl_s$-action (the morphisms in the category are $\pgl_s$-equivariant sheaf morphisms). In particular, for $n$ a natural number, we can interpret $J_n$-sheaves on $\azu^\op$ as sheaves on $\comm^\op$ equipped with a $\pgl_n$-action (every possible action is continuous in this case). Another important case is when we take $s$ to be the maximal supernatural number
\[
s = \prod_{p} p^\infty.
\]
Then the Grothendieck topology $J_s$ is the maximal topology as introduced in \cite{azureps}. So sheaves for the maximal topology also have an interpretation in terms of equivariant sheaves on $\comm^\op$.

\section{Grothendieck topologies on the big cell}
\label{topologies-on-big-cell}

One poset that is closely related to the category of Azumaya algebras is the poset of natural numbers with (the opposite of) the division relation. For this category, we take as objects the strictly positive natural numbers, and a unique morphism $m \to n$ whenever $n \mid m$. We will call this category the big cell and denote it by $\mathtt{D}$.

In this section, we will give a classification of all Grothendieck topologies on $\mathtt{D}$ that have enough points, and determine which of these correspond to coherent subtoposes. The latter are precisely the Grothendieck topologies such that every covering sieve contains a finitely generated covering sieve, so these are the ones we need when constructing Grothendieck topologies on the category of Azumaya algebras.

The starting point is the equivalence 
\begin{equation}
\widehat{\mathtt{D}} \simeq \sh(\SS)
\end{equation}
from \cite[Theorem 1]{llb-covers}, where $\SS$ is the set of supernatural numbers, equipped with the \emph{localic topology}, which has as open subsets the ideals
\begin{equation}
(n_i)_{i \in I} = \{ s \in \SS : \exists i \in I \text{ with }n_i \mid s \}.
\end{equation}
Because $\mathtt{D}$ is a small category, the Grothendieck topologies on $\mathtt{D}$ correspond precisely to the subtoposes of $\widehat{\mathtt{D}}$, so they correspond to \emph{sublocales} of $\SS$, see \cite[IX.5, Corollary 6]{mm-sheaves}. In particular, the Grothendieck topologies with enough points correspond to the \emph{spatial sublocales}, which in turn are in bijective correspondence with the sober subspaces $S \subseteq \SS$, by \cite[IX.3, Corollary 4]{mm-sheaves}. So in order to classify the topologies with enough points, we give a characterization of the subspaces $S \subseteq \SS$ that are sober.

By \cite[C1, Lemma 1.2.5]{elephant-2}, $S$ is sober if and only if it satisfies the following property:
\begin{align} \label{criterion}
V \cap S \neq \varnothing \text{ for all locally closed }V\subseteq\SS\text{ containing }s \Rightarrow s \in S.
\end{align}
So we first give a description of the locally closed sets. Take $s \in \SS$ a supernatural number and $n \mid s$ a natural number. Then
\[
(n) \cap \overline{\{s\}}
\]
is locally closed. Conversely, for any locally closed set $V$ containing $s$, we can find a natural number $n \mid s$ such that
\[
(n) \cap \overline{\{s\}} \subset V.
\]
If we take $s,s' \in \SS$ and natural numbers $n \mid s$ and $n' \mid s'$, then
\[
\left((n) \cap \overline{\{s\}} \right)\cap \left((n') \cap \overline{\{s'\}}\right) = (\lcm(n,n')) \cap \overline{\{\gcd(s,s')\}},
\]
so this particular kind of subsets is a basis for some topology, which we will call the pcfb-topology (primewise convergence from below).

\begin{proposition}
A subspace $S \subseteq \SS$ is sober if and only if it is pcfb-closed.
\end{proposition}
\begin{proof}
By the criterion (\ref{criterion}), $S$ being sober is equivalent to the following statement: if $(n)\cap\overline{\{s\}}$ intersects $S$ nontrivially for all $n \in \NN_+$ with $n \mid s$, then $s \in S$. But this statement is also equivalent to $S$ being pcfb-closed.
\end{proof}

We say that $(s_k)_{k \in \NN}$ \emph{pcfb-converges} to a supernatural number $s$ if and only if
\begin{itemize}
\item $s_k \mid s$ for all $k \in \NN$;
\item for all $n \in \NN_+$ with $n \mid s$ there is a $k \in \NN$ such that $n \mid s_k$.
\end{itemize}
This notion gives a more concrete description for the pcfb-topology: a set $S$ is pcfb-closed if for any sequence in $S$, the pcfb-limit is again in $S$.

What are the Grothendieck topologies on $\mathtt{D}$ corresponding to a pcfb-closed subset $S \subseteq \SS$? It is enough to find one Grothendieck topology $K_S$ with enough points such that the topos-theoretic points of $\sh(\mathtt{D},J)$ are precisely $S$: we saw above that this topology is in fact unique.

\begin{definition} \label{def:K_S}
We say that a sieve $\{n_i \to n \}_{i \in I}$ \emph{contains} a supernatural number $s$ if and only if there is some $i \in I$ such that $n_i \mid s$. In this way, we will often identify a sieve $\{n_i \to n \}_{i \in I}$ with the open set $(n_i)_{i \in I} \subseteq \SS$.

Now we define $K_S(n)$ to be the set of sieves on $n$ that contain $(n)\cap S$.
\end{definition}

Under the above identification, the pullback of a sieve $L$ on $n$ along a morphism $f : m \to n$ is given by $(m) \cap L$.

\begin{proposition}
$K_S$ as above is a Grothendieck topology on $\mathtt{D}$, and its topos-theoretic points are precisely $S \subseteq \SS$.
\end{proposition}
\begin{proof} We first check the three axioms of Grothendieck topologies.

\begin{enumerate}
\item It is clear that the maximal sieve on $n$ contains $S$.
\item Let $L$ be a sieve on $n$, so $L \supseteq (n)\cap S$, and let $f: m \to n$ be a morphism. Then $(m) \cap L \supseteq (m) \cap S$, so pullbacks of covering sieves are again covering sieves.
\item Let $M$ be a covering sieve on $n$, and let $L$ be a sieve on $n$ such that $(m) \cap L$ is a covering sieve for all $m \in M \cap \NN_+$. Then for any $s \in (n)\cap S$, we know that $s \in (n) \cap M$, so take a natural number $m_0 \in M$ with $n \mid m_0 \mid s$. Then $s \in (m_0) \cap L$ because of the assumption, but this means in particular that $s \in (n)\cap L$. So $L$ is a covering sieve.
\end{enumerate}
The family of points for $\sh(\mathtt{D},K_S)$ is given by the supernatural numbers $s \in \SS$ such that for all $n \mid s$ and $K_S$-covering sieves $L$ on $n$, there exists a natural number $n' \in L$ such that $n' \mid s$, in other words $s \in L$. So clearly all elements of $S$ are topos-theoretic points for the topology $K_S$. Now take $s' \notin S$. Because $S$ is pcfb-closed, we can find a natural number $n \mid s'$ such that $(n) \cap \overline{\{s'\}}$ lies in the complement of $S$. Now $(n) \cap (\SS \setminus \overline{\{s'\}})$ is a $K_S$-covering sieve on $n$ not containing $s'$, so $s'$ is not a topos-theoretic point for $K_S$.
\end{proof}

\begin{corollary}
There is a bijective correspondence
\[
S \mapsto K_S
\]
from pcfb-closed subspaces $S \subseteq \SS$ to Grothendieck topologies on $\mathtt{D}$ with enough points. Under this correspondence, $S$ is the space of points of $\sh(\mathtt{D},K_S)$ and there is a commutative diagram of geometric morphisms
\[
\begin{tikzcd}
\sh(\mathtt{D},K_S) \ar[r,"{\sim}"] \ar[d,hook] & \sh(S) \ar[d,hook] \\
\widehat{D} \ar[r,"{\sim}"] \ar[r] & \sh(\SS)
\end{tikzcd}.
\]
\end{corollary}

We will be particularly interested in Grothendieck topologies $K$ on $\mathtt{D}$ such that every covering sieve contains a finitely generated covering sieve (these are the ones we use for constructing Grothendieck topologies on $\azu^\op$). We will prove that these are precisely the topologies $K_S$ with $S$ a \emph{patch} of $\SS$ (also called coherent subspace or spectral subobject). A counter-example is $K_S$ with $S = \SS \setminus \{1\}$. Then the sieve generated by all primes is a covering sieve on $1$, but it does not contain any other covering sieves, in particular no finitely generated ones.

\begin{definition}[\cite{hochster}]
A \emph{spectral space} is a topological space $X$ such that
\begin{enumerate}[label=(S{\arabic*})]
\item $X$ is $T_0$ and quasi-compact; \label{S1}
\item the quasi-compact open sets form a basis; \label{S2}
\item for $U$ and $V$ quasi-compact open, $U \cap V$ is again quasi-compact; \label{S3}
\item every non-empty irreducible subset has a generic point. \label{S4}
\end{enumerate}
For $X$ a spectral space, a \emph{patch} $Y \subseteq X$ is a subspace that is a spectral space, and such that $U \cap Y$ is a quasi-compact open for every quasi-compact open $U$ in $X$.
\end{definition}

It is easy to see that the quasi-compact opens in $\SS$ correspond to the finitely generated ideals
\[
(n_1,\dots,n_k) = \{ s \in \SS : \exists i \in \{1,\dots,k\} \text{ such that } n_i \mid s \}
\]
for $n_1,\dots,n_k \in \NN_+$. Now it is easy to show that $\SS$ is a spectral space (the $T_0$ property and axiom \ref{S4} follow from the fact that $\SS$ is sober).

\begin{proposition} \label{coherence-and-patches}
The following are equivalent:
\begin{enumerate}
\item Every $K$-covering sieve on some $n$ in $\mathtt{D}$ contains a finitely generated $K$-covering sieve.
\item $K = K_S$ for some patch $S \subseteq \SS$.
\end{enumerate}
\end{proposition}
\begin{proof}
\underline{$(1) \Rightarrow (2)$}. Suppose that every $K$-covering sieve contains a finitely generated $K$-covering sieve. If we show that $\mathtt{D}$ has all finite limits, then $\sh(\mathtt{D},K)$ is a coherent topos, see \cite[IX.11, p.519]{mm-sheaves}. It is enough to show that $\mathtt{D}$ has finite products and equalizers. But finite products are given by taking least common multiples (the terminal object is $1$), and equalizers exist for trivial reasons. Because $\sh(\mathtt{D},K)$ is a coherent topos, it has enough points by Deligne's completeness theorem. So we can find a pcfb-closed subset $S \subseteq \SS$ such that $K = K_S$. 

We claim that $(n) \cap S$ is quasi-compact for all $n \in \NN_+$. Indeed, let $\{U_i\}_{i \in I}$ be a family of open sets in $\SS$ covering $(n)\cap S$. Then for every $s \in (n) \cap S$ there is an
\[
n_s \in \bigcup_{i \in I} U_i
\]
with $n \mid n_s \mid s$ and $n_s \in \NN_+$. The sieve generated by the family $\{n \to n_s\}_{s \in S}$ is a $K_S$-covering sieve, and we can take a finite subset $S_0 \subseteq S$ such that $\{n \to n_s\}_{s \in S_0}$ still generates a $K_S$-covering sieve. By taking for each $s \in S_0$ some $i \in I$ such that $n_s \in U_i$, we can construct a finite subset $I_0 \subseteq I$ such that
\[
\bigcup_{i \in I_0} U_i
\]
still covers $(n) \cap S$.

It is now easy to see that the quasi-compact opens in $S$ are precisely the sets of the form $S \cap U$ with $U$ quasi-compact open in $\SS$. Moreover, we already know that $S$ is a sober space. From these two facts it is easy to prove that $S \subseteq \SS$ is a patch.

\underline{$(2) \Rightarrow (1)$}. Take $L$ a $K_S$-covering sieve on $n$. Then $L$ corresponds to an open set
\[
(n_i)_{i \in I} = \{ s \in \SS : \exists i \in I \text{ such that }n_i \mid s \}
\]
with $n_i \in \NN_+$ for $i \in I$. This can be seen as an infinite union of open sets $(n_i)$ in $X$ covering $(n) \cap S$. Because $(n)\cap S$ is quasi-compact, we can take a finite subset $I_0 \subseteq I$ such that
\[
\bigcup_{i \in I_0} (n_i)
\]
still covers $S$. But then the sieve corresponding to $(n_i)_{i \in I_0}$ is a finitely generated $K_S$-covering sieve on $(n)$.
\end{proof}

Let $X$ be a spectral space. In \cite{hochster}, Hochster considers the topology on $X$ with as a subbase the quasi-compact opens for the original topology and their complements. He calls this the \emph{patch topology}, and shows that the closed sets for this topology coincide with the patches of $X$ (this explains the name ``patch''). So in this way we get a more concrete characterization: the patches are the subsets that can be constructed from closed sets and finitely generated open sets, using finite unions and arbitrary intersections.

\begin{example} These are examples of patches $S \subseteq \SS$.
\begin{enumerate}[ref={\theexample.(\arabic*)}]
\item \emph{\underline{Finitely generated sieves.}} These are the sets of the form $S = (n_1,\dots,n_k)$ with $n_1, \dots, n_k \in \NN_+$. These correspond to open subtoposes and they are all dense, because they contain the generic point
\[
\prod_{p \in \mathcal{P}} p^\infty.
\]
\item \emph{\underline{Closed sets.}} We already know that the irreducible closed sets (for the original localic topology) are the sets of the form
\[
\overline{\{s\}} = \{ x \in \SS : x \mid s \}
\]
for $s$ an arbitrary supernatural number. However, because $\SS$ is not noetherian, there might be closed sets that can not be written as a finite union of point closures. The only dense closed set is of course $\SS$ itself.
\item \emph{\underline{Finite sets.}} \label{example-finite-sets}
Take a sequence of natural numbers $(n_i)_i$ with pcfb-limit $s \in \SS$. Then we can write
\[
\{s\} = \overline{\{s\}} \cap \bigcap_{i \in \NN} (n_i)_i
\]
and this shows that each singleton is a patch. As a corollary, each finite set is a patch.
\item \emph{\underline{Topologies from \cite{azureps}.}} Let $\Sigma \subseteq \mathcal{P}$ be a set of primes, and let $S_\Sigma$ be the set of multiples of
\[
s_{\Sigma} = \prod_{p \in \Sigma} p^\infty.
\]
Then $S_\Sigma$ is a patch because
\[
S_\Sigma = \bigcap_{\substack{p \in \Sigma \\ k \in \NN}} (p^k).
\]
This corresponds to the topology for which $\{n_i \to n \}_{i \in I}$ is a covering sieve if and only if $\frac{n_i}{n}$ has all prime divisors in $\Sigma$ for at least one $i \in I$. The set $S_\Sigma$ is dense for each choice of $\Sigma$. For $\Sigma = \varnothing$ we recover the chaotic (minimal) topology, for $\Sigma = \mathcal{P}$ we recover the atomic (maximal) topology.
\item \emph{\underline{Power set of the primes.}} \label{power-set-of-primes} Again let
\[
s_\Sigma = \prod_{p \in \Sigma} p^\infty
\]
for a set of primes $\Sigma$. Then the subset
\[
2^{\mathcal{P}} = 
\left\{  
s_\Sigma : \Sigma \subseteq \mathcal{P}
\right\}
\]
is a patch, because it can be written as the intersection
\[
2^{\mathcal{P}} = \bigcap_{p \in \mathcal{P},\,k \in \NN_+} \{ s : p^k \mid s \} \cup \{ s : p \nmid s \}.
\]
For a finite subset $X \subset \mathcal{P}$ one can look at
\[
U_X = \left\{ \Sigma \subseteq \mathcal{P} \text{ such that } X \subseteq \Sigma \right\}.
\]
These form a basis of open sets for the subspace topology.

\item \emph{\underline{The ring spectrum $\spec(\ZZ)$.}} For each prime $p$ we can consider the supernatural number
\[
s_p = \prod_{\substack{q \neq p \\ \text{prime}}} q^\infty.
\]
Then the subspace $\{ s_p \mid p \text{ prime} \} \cup \{ \prod_p p^\infty \}$ is homeomorphic to $\spec(\ZZ)$, as observed in \cite[Proposition 6]{llb-covers}. It is a patch because it can be written as
\[
\spec(\ZZ) = 2^{\mathcal{P}} \cap \bigcap_{\substack{p \neq q \\ \text{primes}}} (p,q),
\]
with $2^{\mathcal{P}}$ as above in Example \ref{power-set-of-primes}.
\end{enumerate}
\end{example}

By \cite[Theorem 6]{hochster}, each patch $S \subseteq \SS$ is homeomorphic to the spectrum of some commutative ring $R_S$ and each inclusion of patches $S \subseteq S'$ is induced by a ring morphism $R_{S'} \to R_S$. Note that in our case these rings are often non-noetherian, because the patches are non-noetherian as topological spaces.

Any finite $T_0$-space is Alexandrov-discrete for trivial reasons, so their topology is uniquely determined by the partial order given by specialization.

\begin{lemma}
Let $X$ be a finite $T_0$-space. Then $X$ is homeomorphic to a patch in $\SS$. 
\end{lemma}
\begin{proof}
It is enough to show that any finite poset can be embedded into $\NN_+ \subseteq \SS$, seen as a partial order by defining
\[
n \leq_{\mathrm{d}} m \quad\Leftrightarrow\quad n \mid m.
\]
If a poset can be embedded into $(\NN_+,\leq_{\mathrm{d}})$, then it can also be embedded in
\[
A = \{n \in \NN_+ : 2 \mid n \} \quad\text{and}\quad
B = \{n \in \NN_+ : 2 \nmid n\}
\]
because these are both isomorphic to $(\NN_+,\leq_{\mathrm{d}})$ as posets. Now assume that any poset with $k$ elements can be embedded into $\SS$, and take a poset $P$ with $k+1$ elements. Pick a minimal element $x\in P$ and send it to $2$. Then take a poset embedding of $\{ y : y > x \}$ into $A$, and a poset embedding of $\{ y : y \not> x \}$ into $B$. Together, this gives an embedding of $P$ into $\SS$.
\end{proof}

\begin{remark}
In \cite{topologies-poset}, Grothendieck topologies on arbitrary posets are studied. In particular, it is shown that one can construct a Grothendieck topology $J_X$ for every subset $X$ of the poset. Moreover, these are all the possible Grothendieck topologies if and only if the poset is artinian \cite[Proposition~2.12]{topologies-poset}.

In our case, for a set $X$ of natural numbers, we can look at the pcfb-closure $S$ and then take the Grothendieck topology $K_S$ on $\mathtt{D}$, as in Definition \ref{def:K_S}. Then $K_S = J_X$ with $J_X$ defined as in \cite{topologies-poset}. Because the big cell is not artinian, there should be Grothendieck topologies that are not of this form, and these are indeed easy to find: take for example $S$ a singleton consisting of an infinite supernatural number.
\end{remark}

\begin{remark}
Supernatural numbers classify algebraic extensions of finite fields, or for example UHF-algebras. This can be explained conceptually by interpreting the big cell as a site. Because if a mathematical object can be defined for every natural number, then it is often in fact a presheaf on the big cell $\mathtt{D}$.

In the case of algebraic extensions of finite fields, we can look at a presheaf
\[
F(n) = \mathbb{F}_{q^n}
\]
with restriction morphisms $F(n) \to F(m)$ given by field inclusions. Then for each supernatural number $s \in \SS$ there is a stalk $F_s$, and this stalk corresponds to an algebraic extension of $\mathbb{F}_q$.

Similar constructions will be used in Section \ref{points} and in Section \ref{pgl-action}, when we associate matrix algebras or projective general linear groups to each supernatural number.
\end{remark}

\section{Trivializing Grothendieck topologies on Azumaya algebras}
\label{trivializing-topologies}

In the following, $\comm$ will be the category of finitely generated commutative rings, and $\azu$ will be the category of finitely generated Azumaya algebras and center-preserving algebra morphisms. So if $A$ is a finitely generated Azumaya algebra with center $C$, and $B$ is a finitely generated Azumaya algebra with center $D$, then a morphism in $\azu$ is an algebra morphism
\[
f : A \to B
\]
such that $f(C) \subseteq D$. Each finitely generated Azumaya algebra has a finitely generated center. To see this, note that the center is $\CC[\rep_n A]^{\mathrm{GL}_n}$ and use Hilbert's theorem on invariant subrings. Alternatively, use the (noncommutative) Artin--Tate lemma \cite[Lemma~13.9.10]{robson}. So taking centers gives a functor
\[
\z : \azu \longrightarrow \comm.
\]

Let $J$ be a Grothendieck topology on $\comm^\op$, and let $K$ be a Grothendieck topology on $\mathtt{D}$. Recall from \cite{azureps} that a sieve $S$ on an Azumaya algebra $A$ is called a \emph{$J_K$-covering sieve} if
\[
\pi_K(S) = \{ f : C \to D \text{ such that } f \text{ is centrally covered w.r.t.\ }K\text{ and }S \}
\]
is a $J$-covering sieve. Here $f$ is \emph{centrally covered} if $A \otimes_C D$ is of constant degree $n$ and there is a $K$-covering sieve consisting only of numbers $m \in \NN_+$ that are \emph{represented} on $f$, i.e.\ such that $S$ contains a central extension of $A \otimes_C D$ that is of degree $m$ over its center $D$.

Suppose that every $K$-covering sieve can be refined to a finitely generated covering sieve, and that we are moreover in the following two cases:
\begin{enumerate}
\item $J$ is finer than the \'etale topology, or
\item $J$ is finer than the Zariski topology and $K$ is stable under multiplication.
\end{enumerate}
Then $J_K$ is a Grothendieck topology on $\azu^\op$, see \cite[Theorem~2.7]{azureps}. So in this situation we will call $J_K$ a \emph{combined Grothendieck topology}.

Section \ref{topologies-on-big-cell} allows us to reformulate some of the above conditions. In Proposition \ref{coherence-and-patches} we characterized the Grothendieck topologies $K$ on $\mathtt{D}$ such that all $K$-covering sieves contain a finitely generated $K$-covering sieve: these are precisely the Grothendieck topologies of the form $K=K_S$ for a \emph{patch} $S \subseteq \SS$. Moreover, it is easy to see that $K$ is stable under multiplication if and only if $ks \in S$ implies $s \in S$ for all $k \in \NN_+$ and $s \in \SS$.

For some of the Grothendieck topologies $J_K$ one could say that all Azumaya algebras are $J_K$-locally given by matrix algebras. We can make this idea precise in a few different ways, that turn out to be equivalent.  

\begin{proposition}
Let $J_K$ be a combined Grothendieck topology. Then the following are equivalent:
\begin{enumerate}
\item for every Azumaya algebra $A$, there is a family of morphisms
\[
\{ A \to \M_{n_i}(C_i) \}_{i \in I}
\]
generating a $J_K$-covering sieve;
\item for every Azumaya algebra $A$, there is a family of morphisms
\[
\{ A \to \M_{n_i}(C_i) \}_{i \in I}
\]
such that the induced morphism
\[
\bigsqcup_{i \in I} \azu(\M_{n_i}(C_i), -) \longrightarrow \azu(A,-)
\]
is a $J_K$-epimorphism (after sheafification);
\item for every Azumaya algebra $A$ and $J_K$-covering sieve $S$ on $A$, the sieve $S$ contains a $J_K$-covering sieve generated by a family
\[
\{ A \to \M_{n_i}(C_i) \}_{i \in I}.
\]
\end{enumerate}
\end{proposition}
\begin{proof}
$(3) \Rightarrow (1) \Leftrightarrow (2)$ This is clear.

$(1) \Rightarrow (3)$ Suppose $S = \{ A \to A_i \}_{i \in I}$. Then take for each $A_i$ a family $$\{ A_i \to \M_{n_{ir}}(C_{ir}) \}_{r \in R_i}$$ generating a $J_K$-covering sieve. Then $\{ A \to \M_{n_{ir}}(C_{ir}) \}$ generates a $J_K$-covering sieve on $A$.
\end{proof}

\begin{definition}
A combined Grothendieck topology $J_K$ satisfying the equivalent conditions above will be called \emph{trivializing}.
\end{definition}

Note that $J_K$ is always trivializing whenever $J$ is finer than the \'etale topology. If $J$ is the Zariski topology, then we need more restrictions on the Grothendieck topology $K$. 

Indeed, let $A = (x,y)_n$ be the cyclic algebra on $C=\CC[x,y,x^{-1},y^{-1}]$, then it is of degree $n$ and of period $n$ (the period is the order in the Brauer group). Moreover, because $\CC[x,y,x^{-1},y^{-1}]$ is the coordinate ring of a nonsingular and connected variety, the period of $A$ is still $n$ after Zariski localization: the restriction morphism $\wis{Br}(\CC[x^\pm,y^\pm]) \to \wis{Br}(\CC(x,y))$ is injective. So if $f: C \to D$ is some Zariski localization, and $A \otimes_C D \to \M_k(D)$ is a center-preserving morphism, then $\M_k(D) \cong (A \otimes_C D) \otimes_D B$ for some Azumaya algebra $B$. This $B$ is of opposite Brauer class, so it has again period $n$. Because the period divides the degree, we find $n^2 \mid k$. So let $K = K_S$ for some patch $S \subseteq \SS$. Then in order for $\textrm{Zar}_K$ to be trivializing, we need $n \mid s \Rightarrow n^2 \mid s$ for all $s \in S$. In other words, $S$ contains only supernatural numbers of the form
\[
s_\Sigma = \prod_{p \in \Sigma} p^\infty
\]
with $\Sigma$ some set of primes. In the special case $\Sigma = \varnothing$ we set $s_\varnothing = 1$. The supernatural numbers of the form $s_\Sigma$ will be called \emph{completely infinite}.

Conversely, suppose that $K=K_S$ with $S \subseteq \SS$ a patch consisting of only completely infinite supernatural numbers. Then for all $n,k \in \NN_+$ we have that $\{n^k \to n\}$ generates a $K$-covering sieve. We will use a result of Bass which says that every Azumaya algebra of degree $n$ can be embedded into a matrix algebra of degree $n^k$ for some $k$, see Proposition 6.1 and the remark afterwards in \cite{bass-lectures}. Now take $A$ an Azumaya algebra with center $C$ and let $f : C \to D$ be a morphism such that $A \otimes_C D$ is of constant degree $n$. Then we can take a central extension $A \otimes_C D \subseteq \M_{n^k}(D)$ for some $k \in \NN_+$. The sieve generated by these central extensions for each such $f$ is clearly a $J_K$-covering sieve, so $J_K$ is trivializing. We summarize the above in the following proposition.

\begin{proposition}
Let $\mathrm{Zar}$ be the Zariski topology and consider $K=K_S$ as in Section \ref{topologies-on-big-cell}, with $S \subseteq \SS$ a patch. Then
$\mathrm{Zar}_K$ is trivializing if and only if $S$ only contains \emph{completely infinite} supernatural numbers, i.e.\ if and only if each $s \in S$ can be written as
\[
s = \prod_{p \in \Sigma} p^\infty
\]
for some set $\Sigma$ of primes.
\end{proposition}

For any patch $S \subseteq \SS$, the subset of completely infinite elements is again a patch: it is the intersection with $2^\mathcal{P}$ from Example \ref{power-set-of-primes}.

\section{Topos-theoretic points}
\label{points}

Points of a topos are usually defined as certain flat functors or geometric morphisms, but they can also be interpreted as pro-objects with respect to the site. This is proven in \cite[Prop.\ 7.13]{johnstone}, under the assumption that the site has fiber products. This assumption does not hold for $\azu^\op$ (for an explicit proof, see \cite[Section~2]{azureps}). Fortunately, as remarked by Gabber and Kelly in \cite{gabber-kelly}, the fiber product condition is not necessary.

This means that in our case, the points of $\psh(\azu^\op)$ correspond precisely to the ind-objects in $\azu$. Moreover, this category of ind-objects can be embedded fully faithfully into the category $\zalg$ of algebras and center-preserving algebra morphisms. Indeed, we have a natural isomorphism
\[
\zalg(\varinjlim_i A_i, \varinjlim_j B_j)
\simeq \varprojlim_i \varinjlim_j \azu(A_i,B_j),
\]
because each $A_i$ and $B_j$ is finitely generated.

One notable area where inductive limits of matrix algebras play a role, is the theory of $C^*$-algebras. Recall that an UHF-algebra is the closure of
\[
\bigcup_{i \in \NN} \M_{n_i}(\CC)
\]
for some chain
\[
\M_{n_1}(\CC) \to \M_{n_2}(\CC) \to \dots \to \M_{n_i}(\CC) \to \dots
\]
where the morphisms are algebra morphisms (so they are injective). It is well-known that UHF-algebras are classified by their associated supernatural number $s$, which is uniquely determined by
\[
n \mid s \Leftrightarrow \exists i\in\NN \text{ with } n \mid n_i.
\]
One can check that the same classification holds for the algebras of the form $\bigcup_{i \in \NN} \M_{n_i}(\CC)$, i.e.\ without taking the closure, so we will denote these by $\M_s(\CC)$, where $s$ is the associated supernatural number.

\begin{remark}
Note that not every ind-Azumaya with center $\CC$ is a union of a countable chain of matrix algebras. If $\{V_i\}_{i \in I}$ is an infinite set of finite dimensional vector spaces, then the matrix algebras
\[
\en\left( \bigotimes_{j \in J} V_j \right)
\]
for $J \subset I$ finite, form a filtered diagram. It is easy to see that these often do not correspond to UHF-algebras.
\end{remark}

An ind-Azumaya algebra $B$ is a point for $\sh(\azu^\op,J_K)$ if and only if it is $J_K$-local, i.e.\ if and only if for each $A \to B$ and covering sieve $\{A \to A_i \}_{i \in I}$ there is an $i \in I$ and a factorization
\begin{equation*} \label{eg:local-for-JK}
\begin{tikzcd}
 & A_i \ar[rd,dashed] & \\
 A \ar[ru] \ar[rr] & & B
\end{tikzcd}.
\end{equation*}
But before we can say anything about the points of the Azumaya toposes, we need the following proposition, which is straightforward but was not found in the literature.

\begin{proposition} \label{prop:topologies-on-comm}
Let $J$ be a Grothendieck topology on $\comm^\op$ with enough points. Take $C$ in $\comm^\op$ and let $L$ be a sieve on $C$. Then $L$ is a $J$-covering sieve if and only if for every $J$-local ring $D$ and morphism $C \to D$ there is some $C \to C'$ in $L$ and a factorization
\[
\begin{tikzcd}
 & C' \ar[rd,dashed] & \\
 C \ar[ru] \ar[rr] & & D
\end{tikzcd}.
\]
\end{proposition}
\begin{proof}
The ``only if'' part follows by definition of $J$-local ring, so we only need to prove the ``if'' part.

For $L = \{ C \to C_i \}_{i \in I}$, consider the morphism
\[
\bigsqcup_{i \in I} \comm(C_i, -) \stackrel{\xi}{\longrightarrow} \comm(C,-).
\]
By definition, the stalk of $\comm(C,-)$ at a $J$-local ring $D$ is given by $\varinjlim_j \comm(C,D_j)$ where $D = \varinjlim_j D_j$ for $(D_j)_j$ is a filtered system of finitely generated commutative rings. But because $C$ is itself finitely generated, we have an isomorphism $\varinjlim_j \comm(C,D_j) \simeq \comm(C,D)$. So the assumption implies that $\xi$ is an epimorphism (this can be checked on stalks because $J$ has enough points). In particular, the identity morphism $C \to C$ is locally in the image of $\xi$, i.e.\ we can take a $J$-covering sieve $M$ contained in the image of $\xi$. It is clear that $L$ contains $M$, so $L$ is a $J$-covering sieve too.
\end{proof}

Note that in the above proposition we explicitly use that the objects of the site are \emph{finitely generated} commutative rings.

\begin{theorem} \label{thm:points}
Let $J_K$ be a trivializing combined Grothendieck topology, so $K=K_S$ for some patch $S \subseteq \SS$.
Then
\[
\mathrm{P}(J,K) = \{ \M_s(D) \mid s \in S \text{ and }D\text{ a }J\text{-local commutative algebra} \}
\]
is a family of points for $\sh(\azu^\op,J_K)$. If $J$ has enough points, then this is a separating family of points for $J_K$, so in this case $J_K$ has enough points too.
\end{theorem}
\begin{proof}
We first show that each $\M_s(D)$ as above is $J_K$-local. Let $A$ be an Azumaya algebra and take a morphism $A \to \M_s(D)$. Let $L$ be a $J_K$-covering sieve on $A$. We can assume that $L$ is generated by matrix algebras. Then the sieve of centrally covered morphisms $\pi_K(L)$ is a $J$-covering sieve on the center $C$ of $A$. In particular, there is a centrally covered morphism $C \to C'$ and a factorization
\[
\begin{tikzcd}
 & C' \ar[rd,dashed] & \\
 C \ar[ru] \ar[rr] & & D
\end{tikzcd}.
\]
Moreover, we can find a central extension $A \otimes_C C' \to A'$ with $A'$ of degree $m \mid s$. Because $L$ is generated by matrix algebras, there is a commutative diagram
\[
\begin{tikzcd}
\M_{m'}(C'') \ar[r] & A' \\
A \ar[u] \ar[r] & A \otimes_C C' \ar[u]
\end{tikzcd}
\]
with $m' \mid m \mid s$. The desired factorization is now given by a choice of dashed arrow
\[
\begin{tikzcd}
 & \M_{m'}(C'') \ar[rd,dashed] & \\
 A \ar[ru] \ar[rr] & & \M_s(D)
\end{tikzcd},
\]
which on centers is given by the composition $C'' \to C' \to D$.

It remains to show that the given family of points $\mathrm{P}(J,K)$ is separating whenever $J$ has enough points. Take a morphism of sheaves $\phi: \mathcal{F} \to \mathcal{G}$ inducing isomorphisms on stalks for all points in $\mathrm{P}(J,K)$. Consider $x,y \in \mathcal{F}(A)$ with $\phi(x) = \phi(y)$, with $A$ a finitely generated Azumaya algebra. We claim that $x=y$. Let $L$ be the sieve of morphisms $f: A \to A'$ such that $x=y$ after restriction along $f$. Then it is enough to show that $L$ is a $J_K$-covering sieve. In particular, we can assume that $A$ is of constant degree $n$.

Let $\M_s(D)$ be a $J_K$-local ring in $\mathrm{P}(J,K)$, and take a morphism $g : A \to \M_s(D)$. Write $\M_s(D) = \varinjlim_j B_j$ with $(B^s_j)_j$ a filtered system of (finitely generated) Azumaya algebras. We can find a $j$ such that $g(A) \subseteq B^s_j$ and such that $g: A \to B^s_j$ is in $L$. Because $D$ is Zariski-local, we can assume that $B^s_j$ is of constant degree $n_s \mid s$ over its center $D_s$. We can do this for every $s$, and then take a finite subset $S' \subseteq S$ such that $\{ n_s \}_{s \in S'}$ generates a $K$-covering sieve. Let $D'$ be the tensor product over $C$ of the centers $D_s$ with $s \in S'$. Then the morphism $C \to D'$ is centrally covered w.r.t.\ $L$, and there is a factorization
\[
\begin{tikzcd}
 & D' \ar[rd,dashed] & \\
 C \ar[ru] \ar[rr] & & D
\end{tikzcd}.
\]
If the morphism $C \to D$ was arbitrary, this would prove that $\pi_K(L)$ is a $J$-covering sieve, by Proposition \ref{prop:topologies-on-comm}. So we still need to prove that for every $h: C \to D$ with $D$ $J$-local and for every $s \in S$, we can find a morphism $g: A \to \M_s(D)$ inducing $h$ on centers. Take a family of morphisms $$\{ A \to \M_{n_i}(C_i) \}_{i \in I}$$ generating a $J_K$-covering sieve. Then for an arbitrary $h:C \to D$ we can find a centrally covered morphism $C \to C'$ and a factorization
\[
\begin{tikzcd}
 & C' \ar[rd,dashed] & \\
 C \ar[ru] \ar[rr,"{h}"] & & D
\end{tikzcd}.
\]
From this, it is easy to construct a morphism $A \to \M_s(D)$ inducing the given morphism on centers.

In order to prove that $\phi$ is surjective, we take $y \in \mathcal{G}(A)$ and consider the sieve $L$ of morphisms $f: C \to C'$ such that $y$ is in the image of $\phi$ after restriction along $f$. In the same way as above, we can show that $L$ is a $J_K$-covering sieve. Because $\phi$ is injective, the preimages are unique so they can be glued to a preimage of $y$.
\end{proof}

Note that the $J$-local commutative algebras are known in a lot of interesting cases, including the Zariski and \'etale topology. For an overview, we refer to \cite[Table 1]{gabber-kelly}. In \cite[Lemma 3.3]{gabber-kelly}, some properties of points for the flat (fppf) topology are given, but this case is more difficult and still lacks a concrete description. Schr\"{o}er in \cite{schroeer} also studies points for the fppf topology, but he considers the category of arbitrary commutative rings instead of only finitely generated ones. 

Suppose that every finitely generated commutative algebra $C$ is quasi-compact with respect to some Grothendieck topology $J$, i.e.\ any $J$-covering sieve can be refined to a finitely generated $J$-covering sieve. Then $\sh(\comm^\op,J)$ is a coherent topos, so it has enough points by Deligne's completeness theorem. Examples are the Zariski, Nisnevich, \'etale and flat topology.

As in Theorem \ref{thm:points}, let $J_K$ be a trivializing combined Grothendieck topology, $K = K_S$ for some patch $S \subseteq \SS$. If $J_K$ is moreover coarser than the maximal flat topology, then $\alg(R,-)$ is a $J_K$-sheaf for every algebra $R$, see \cite[Proposition~3.2]{azureps}. The slice category \[\sh(\azu^\op,J_K) \slash \alg(R,-)\] has as objects $J_K$-sheaves $\mathcal{F}$ equipped with a structure morphism \[\mathcal{F} \longrightarrow \alg(R,-)\] and as morphisms the sheaf morphisms $\phi: \mathcal{F} \longrightarrow \mathcal{G}$ such that the diagram
\[
\begin{tikzcd}
\mathcal{F} \ar[rr,"{\phi}"] \ar[rd] & & \mathcal{G} \ar[ld] \\
& \alg(R,-) &
\end{tikzcd}
\]
commutes. This is again a topos, see \cite[IV.7, Theorem 1]{mm-sheaves}, and we would like to determine topos-theoretic points of this slice topos.

In general, let $\mathcal{T}$ be a Grothendieck topos and let $\mathcal{F}$ be a sheaf in $\mathcal{T}$. For a family of points $P(\mathcal{T})$, we can construct a point of $\mathcal{T}/\mathcal{F}$ from a point $p \in P(\mathcal{T})$ together with an element $x \in \mathcal{F}_p$ in the stalk. The corresponding stalk is given by
\[
\mathcal{G}_{(p,x)} = \phi_p^{-1}(x)
\]
where $\phi: \mathcal{G} \to \mathcal{F}$ is the structure morphism and $\phi_p$ is the induced map on the stalk. If we start from a separating family of points for $\mathcal{T}$, then it is easy to see that we end up with a separating family of points for $\mathcal{T}/\mathcal{F}$.

In the case of an (finitely generated, not necessarily commutative) algebra $R$, we obtain a family of points
\[
\mathrm{P}_R(J,K) = \{ R \to \M_s(D) \mid s \in S \text{ and }D\text{ a }J\text{-local commutative algebra} \},
\]
using Theorem \ref{thm:points}. Here we assume that $J_K$ is a trivializing combined Grothendieck topology, coarser than the maximal flat topology, with $S$ the patch associated to $K$. If $J$ has enough points, then the set $\mathrm{P}_R(J,K)$ is separating.

\section{Sheaves with action of the projective general linear group}
\label{pgl-action}

Let $s \in \SS$ be a supernatural number. Then the singleton $\{s\} \subseteq \SS$ is a patch, and we can look at the corresponding Grothendieck topology $K_s$. We will use the shorthand $J_s$ for the Grothendieck topology $J_{(K_s)}$ on $\azu^\op$. We will always assume that $J_s$ is a trivializing combined Grothendieck topology. By the results of Section \ref{trivializing-topologies}, this is the case for example when
\begin{enumerate}
\item $J$ is finer than the \'etale topology and $s$ is arbitrary, or
\item $J$ is finer than the Zariski topology and $s$ is completely infinite, i.e.\ it is of the form
\[
s = \prod_{p \in \Sigma} p^\infty
\]
for $\Sigma$ some set of primes.
\end{enumerate}
These are the two cases we keep in mind for this section. Moreover, we will often assume that every $J$-covering sieve contains a finitely generated $J$-covering sieve. 

Now consider $A$ a finitely generated Azumaya algebra of constant degree $n$ over its center $C$.
\begin{enumerate}
\item If $n \nmid s$, then every sieve on $A$ is a $J_s$-covering sieve, including the empty sieve.
\item If $n \mid s$, then a sieve $L$ on $A$ is a covering sieve if and only if there is some $J$-covering $\{f_i: C \to C_i \}_{i \in I}$ such that for each $i \in I$, the sieve $L$ contains a composition
\[
\begin{tikzcd}
 & B_i \\
A \ar[r,"{1 \otimes f_i}"] & A \otimes_C C_i \ar[u,"{\phi_i}"]
\end{tikzcd}
\]
with $\phi_i$ some central extension, $B_i$ of degree $m \mid s$ over its center $C_i$.
\end{enumerate}
One case was already introduced in \cite{azureps}: the maximal topology $J_\max$ associated to a Grothendieck topology $J$ on $\comm^\op$. This is precisely $J_s$ for
\[
s = \prod_{p \in \mathbb{P}} p^\infty
\]
the maximal supernatural number. By \cite[Proposition 3.2]{azureps}, the functor
\begin{gather*}
\alg(R,-) : \azu \longrightarrow \sets \\
A \mapsto \alg(R,A)
\end{gather*}
is a sheaf for the maximal flat topology on $\azu^\op$, i.e.\ for $J_\max$ with $J$ the flat topology (or any coarser topology). Similarly, the representable presheaves $\azu(A,-)$ are sheaves for the maximal flat topology (because morphisms are center-preserving if and only if they are center-preserving $J_\max$-locally). In particular, $J_\max$ is a subcanonical Grothendieck topology for $J$ the flat topology.

We will now show that we can easily construct $J_s$-sheaves from $J_\max$-sheaves, with $s$ any supernatural number.

Let $A$ be a finitely generated Azumaya algebra with center $C$. Then the degree of $A$ is locally constant over $\spec(C)$, so there is a unique decomposition
\[
C = C_1 \times \dots \times C_k
\]
into components, such that $A\otimes_C C_i$ is of constant rank $d_i$ for $i = 1,\dots,k$ and such that $d_i \neq d_j$ for $i \neq j$. If $s$ is a supernatural number, we then define the\emph{ $s$-truncation of $A$} to be the Azumaya algebra
\[
A_s = \prod_{d_i \mid s} A \otimes_C C_i
\]
(the empty product gives the zero ring). Similarly, we define the \emph{$s$-truncation of a presheaf $\mathcal{F}$} to be
\[
\mathcal{F}_{\downarrow s}(A) = \mathcal{F}(A_s).
\]
One can check that this is again a presheaf.
\begin{proposition}
Let $\mathcal{F}$ be a $J_\max$-sheaf, where $J_\max$ is the topology associated to $J$ and the maximal supernatural number
\[
\prod_{p \in \mathbb{P}} p^\infty.
\]
Let $s$ be a supernatural number. Then $\mathcal{F}_{\downarrow s}$ is a $J_s$-sheaf, and it is naturally isomorphic to the $J_s$-sheafification of $\mathcal{F}$.
\end{proposition}
\begin{proof}
Let $A$ be an Azumaya algebra of constant rank $n$. If $n \nmid s$, then $$\mathcal{F}_{\downarrow s}(B) = \{ \ast \}$$
whenever there is a morphism $A \to B$. So, for trivial reasons, $\mathcal{F}_{\downarrow s}$ satisfies the gluing criterion with respect to $J_s$-covering sieves on $A$. If $n \mid s$, then any $J_s$-covering sieve on $A$ can be refined to a covering sieve generated by Azumaya algebras with degrees dividing $s$. Such a sieve is a $J_\max$-covering sieve as well, so here the gluing criterion is satisfied because $\mathcal{F}$ is a $J_\max$-sheaf. If $A$ is not of constant rank, then we can take a Zariski cover on which it is. So this case reduces to the previous two cases.

This shows that $\mathcal{F}_{\downarrow s}$ is a $J_s$-sheaf. Conversely, let $\mathcal{G}$ be a $J_s$-sheaf. The morphism $\pi: A \to A_s$ generates a $J_s$-covering sieve, so the induced map $\mathcal{G} \to \mathcal{G}_{\downarrow s}$ is an isomorphism. Now it is easy to show that $\mathcal{F} \mapsto \mathcal{F}_{\downarrow s}$ is the sheafification functor.
\end{proof}

The next goal of this section is to describe the topos
\[
\sh(\azu^\op,J_s)
\]
in terms of sheaves on $\comm^\op$ equipped with an action of an (infinite) projective linear group $\pgl_s$, that we will define below. The inspiration for this equivalence comes from \cite{caramello-topological}, where atomic toposes are described in terms of a topological group, under some mild assumptions. Because $\sh(\azu^\op,J_s)$ is not an atomic topos, the results from \cite{caramello-topological} do not easily transfer to our setting. However, the general principle still works, and this will result in Theorem \ref{thm:pgl}.

For each $s \in \SS$, we would like to define $\pgl_s$ as
\[
\pgl_s(C) = \bigcup_{n \mid s}\aut_C(\M_n(C)),
\]
but we have to take a uniform choice as to how the automorphism groups are embedded into each other. So we define a functor
\[
\M : \comm \times \mathtt{D}^\op \longrightarrow \azu
\]
as follows. For each $n \in \NN_+$, consider the matrix algebra
\[
\M(C,n) = \M_{p_1}(C) \otimes_C \dots \otimes_C \M_{p_k}(C)
\]
with $p_1 \leq \dots \leq p_k$ the prime factors of $n$, each occurring with the right multiplicity. Moreover, for each $n \mid m$ and morphism $C \to D$, we consider the morphism
\[
\M(C,n) = \M_{p_1}(C) \otimes_C \dots \otimes_C \M_{p_k}(C) \stackrel{\rho_{n,m}}{\longrightarrow} \M_{q_1}(D) \otimes_D \dots \otimes_D \M_{q_l}(D) = \M(D,m)
\]
which is given by sending the $k$-th occurrence of a tensor factor $\M_p(C)$ on the left to the $k$-th occurrence of $\M_p(D)$ on the right side. It is easy to see that this turns $\M$ into a functor. In the following, we make the identification $\M_n(C) = \M(C,n)$, and the inclusion $\rho_{n,m}$ will be called a \emph{standard embedding}.

Note that $\rho_{n,m}$ as above induces a morphism
\[
\aut_C(\M_n(C)) \longrightarrow \aut_C(\M_m(D)).
\]
We then define the projective linear group of degree $s$ as the union
\begin{equation}
\pgl_s(C) = \bigcup_{n \mid s} \aut_C(\M_n(C)).
\end{equation}
Note that there is an obvious action of $\pgl_s(C)$ on the infinite matrix algebra
\begin{equation} \label{M_s}
\M_s(C) = \bigcup_{n \mid s} \M_n(C).
\end{equation}
Using this action, we can define equivalence relations $\sim_n$ on $\pgl_s(C)$ by setting
\[
g \sim_n h \quad\Leftrightarrow\quad g\cdot x = h \cdot x \quad \text{for all }x \in \M_n(C).
\]

\begin{definition} \label{def:continuous}
Let $\mathcal{G}$ be a presheaf on $\comm^\op$, equipped with an action of $\pgl_s$. Then we say that the action is \emph{continuous} if for every finitely generated commutative algebra $C$ and $x \in \mathcal{G}(C)$, there is some $n \in \NN_+$ such that
\[
g \cdot x|_D = x|_D
\]
for every morphism $C \to D$ and $g \in \pgl(D)$ with $g \sim_n 1$.
\end{definition}

It is well-known that $\pgl_n$ for $n \in \NN_+$ is representable by an affine scheme, so it is a sheaf for the canonical Grothendieck topology. Similarly, let $J$ be a subcanonical topology, such that every $J$-cover has a finite subcover. Then we claim that $\pgl_s$ is a $J$-sheaf. It is enough to check the sheaf condition on a finite cover $\{C \to C_i \}_{i = 1}^k$. Take elements $g_i \in \pgl_s(C_i)$ agreeing on intersections. We can assume that each $g_i \in \pgl_n(C_i)$ for some common $n \in \NN_+$. Now there is a unique glued element $g \in \pgl_n(C) \subseteq \pgl_s(C)$. We can prove in an analogous way that $\M_s$ from (\ref{M_s}) is a $J$-sheaf. Moreover, it is easy to see that the action of $\pgl_s$ is continuous (as in Definition \ref{def:continuous}).

More generally, we can take any finitely generated noncommutative algebra $R$ over $\CC$, and look at the functor
\[
\comm \to \sets \qquad D \mapsto \alg(R,\M_n(D)),
\]
sending a commutative ring $D$ to the set of algebra morphisms $R \to \M_n(D)$. It is well-known that this functor is representable by an affine scheme $\rep_n R$. The standard embeddings $\M_n(D) \to \M_m(D)$ induce a morphism of schemes $\rep_n R \to \rep_m R$. So we can again define a functor
\begin{equation} \label{rep_s}
(\rep_s R)(D) = \bigcup_{n \mid s} (\rep_n R)(D)
\end{equation}
and with the same proof as for $\pgl_s$ we can see that this is a $J$-sheaf (under the same conditions: $J$ subcanonical and every $J$-cover has a finite subcover).

We would like to define a sheaf of \emph{trace preserving} representations $\trep_s R$ as well, for $R$ an algebra with trace. For $R$ an algebra, recall that a \emph{trace} on $R$ is a $\CC$-linear function $\tr : R \to R$ such that for all $a,b \in R$
\begin{enumerate}
\item (Maps into center) $\tr(a)b = b\,\tr(a)$;
\item (Necklace property) $\tr(ab) = \tr(ba)$;
\item (Linear with respect to traces) $\tr(\tr(a)b) = \tr(a)\tr(b)$.
\end{enumerate}
Each Azumaya algebra $A$ can be equipped with a trace, by viewing it as a subalgebra of some matrix algebra If $A$ is of constant degree $n$, then $\tr(1) = n$. Now let $\phi : A \to B$ be a center-preserving morphism, with $A$ of constant degree $n$ and $B$ of constant degree $nk$. Then it is easy to see that
\[
\phi(\tr(a)) = \frac{1}{k}\tr(\phi(a)),
\]
so $\phi$ is trace-preserving if and only if $A$ and $B$ are of the same degree. To avoid this, we define the \emph{normalized trace} of an Azumaya algebra $A$ as
\[
\tr'(a) = \frac{1}{n} \tr(a)
\]
for all $a \in A$, with $n$ the degree of $A$. In particular $\tr'(1) = 1$. If $A$ does not have constant degree, then the degree is at least locally constant, so we can define the normalized trace locally. For the normalized trace, morphisms between Azumaya algebras are trace-preserving if and only if they are center-preserving. Moreover, there is an obvious normalized trace on $\M_s(C)$ (well-known in the context of $C^*$-algebras).

If we work with normalized traces, it makes sense to look at the subsheaf
\[
\trep_s(R) \subseteq \rep_s(R)
\]
consisting of the morphisms $R \to \M_s(C)$ that are trace-preserving. In particular, for $A$ an Azumaya algebra, the sheaf $\trep_s(A)$ is given by the center-preserving morphisms $A \to \M_s(C)$ (these are all embeddings).

\begin{proposition} \label{prop:homogeneous}
Take a supernatural number $s \in \SS$. Let $A$ be a finitely generated Azumaya algebra of degree $n \mid s$ over its center $C$. Then $\trep_s A$ is an \emph{\'etale $\pgl_s$-homogeneous space} over $\spec(C)$, in the sense that
\begin{enumerate}
\item \emph{($\exists$ section locally)} we can find an \'etale covering $\{f_i: C \to C_i \}_{i \in I}$ and morphisms $\phi_i : A \to \M_s(C_i)$ extending $f_i$, for all $i \in I$;
\item \emph{(sections in same orbit locally)} for any $f: C \to E$ and $\phi, \psi : A \to \M_s(E)$ extending $f$, we can find an \'etale covering $\{f_i : E\to E_i \}_{i \in I}$ and automorphisms $g_i \in \pgl_s(E_i)$, such that $g_i \cdot \phi|_{E_i} = \psi|_{E_i}$.
\[
\begin{tikzcd}[row sep=tiny,column sep=huge]
& \M_s(E_i) \ar[dd,"{g_i}"] \\
A \ar[ru,"{\phi|_{E_i}}"] \ar[rd,"{\psi|_{E_i}}"'] & \\
& \M_s(E_i)
\end{tikzcd}
\]
\end{enumerate}
In the case that $s$ is completely infinite, $\trep_s(A)$ is even a \emph{Zariski $\pgl_s$-homogeneous space}, i.e.\ the coverings can be Zariski coverings in the two conditions above.
\end{proposition}
\begin{proof}
\begin{enumerate}
\item Just take an \'etale covering $\{ C \to C_i \}_{i \in I}$ trivializing $A$, i.e.\ such that $A \otimes_C C_i \cong \M_n(C_i)$. In the case that $s$ is completely infinite, we can already find an embedding $A \to \M_s(C)$, using \cite[Proposition 6.1]{bass-lectures} and the remark afterwards. So the trivial covering suffices.
\item Take a natural number $m \mid s$ such that $\phi(A)$ and $\psi(A)$ are both contained in $\M_m(E)$. Use the Double Centralizer Theorem to write
\[
\phi(A) \otimes_E B \cong \M_m(E) \cong \psi(A) \otimes_E B'.
\]
Because $\phi(A)$ and $\psi(A)$ are isomorphic, $B$ and $B'$ are in the same Brauer class. Take a Zariski covering $\{E \to E_i\}_{i \in I}$ such that $B \otimes_E E_i \cong B' \otimes_E E_i$. Now we get $g_i \in \pgl_m(E_i)$ by tensoring an isomorphism $\phi(A) \to \psi(A)$ with an isomorphism $B \otimes_E E_i \to B' \otimes_E E_i$.
\end{enumerate}
\end{proof}

If $A$ is an Azumaya algebra of degree $n \nmid s$, then $\trep_s(A)$ is the empty sheaf.

\begin{theorem} \label{thm:pgl}
Take a supernatural number $s \in \SS$. Take a subcanonical Grothendieck topology $J$ such that every $J$-cover admits a finite subcover, and such that $J_s = J_{(K_s)}$ is a trivializing combined Grothendieck topology. Then there is an equivalence of toposes
\[
\sh(\azu^\op,J_s) \simeq \pgl_s - \sh(\comm^\op,J).
\]
Here the right hand side is the topos of $J$-sheaves equipped with a continuous action of $\pgl_s$ (as in Definition \ref{def:continuous}).
\end{theorem}
\begin{proof}%[Proof of Theorem \ref{pgl-sheaves}]
Consider the functors
\[
\begin{tikzcd}
\sh(\azu^\op,J_s) \ar[r,bend left,"{\flat}"] & \pgl_s-\sh(\comm^\op,J) \ar[l,bend left,"{\sharp}"]
\end{tikzcd}
\]
defined by
\begin{gather*}
\mathcal{F}^\flat(C) = \varinjlim_{n \mid s} \mathcal{F}(\M_n(C)) \\
\mathcal{G}^\sharp(A) = \Hom_{\pgl_s}( \trep_s A, \mathcal{G}  ).
\end{gather*}
Here the action of $g \in \pgl_n(C)$ on $\mathcal{F}(\M_n(C))$ is induced by the action on $\M_n(C)$, for all $n \mid s$. It is clear that the action is continuous (as in Definition \ref{def:continuous}).

\underline{$\mathcal{F}^\flat$ is a sheaf.} Analogously to how we proved that $\pgl_s$ is a $J$-sheaf (after Definition \ref{def:continuous}), we can show that $\mathcal{F}^\flat$ is a $J$-sheaf whenever $\mathcal{F}$ is a $J_s$-sheaf. Here we need the assumption that every $J$-cover admits a finite subcover. 

\underline{$\mathcal{G}^\sharp$ is a sheaf.} For $\{ \phi_i: A \to A_{i} \}_{i \in I}$ a $J_s$-covering sieve, the sheaf morphism
\[
\bigsqcup_{i \in I} \trep_s(A_i) \longrightarrow \trep_s A
\]
is an epimorphism, because it is a surjection on stalks. This implies that if $\mathcal{G}$ is a sheaf, then $\mathcal{G}^\sharp$ is a separated presheaf. To show that it is a sheaf, take a matching family of sections
\[
s_i \in \mathcal{G}^\sharp(A_i), \quad\quad i \in I
\]
corresponding to sheaf morphisms
\[
\psi_i : \trep_s A_i \longrightarrow \mathcal{G}.
\]
The fact that $(s_i)_{i \in I}$ is a matching family of sections, translates as follows to a condition on the $\psi_i$'s: if there is a commutative diagram
\[
\begin{tikzcd}[row sep=small]
& A_i \ar[rd] & \\
A \ar[ru,"{\phi_i}"] \ar[rd,"{\phi_j}"'] & & B \\
& A_j \ar[ru] &
\end{tikzcd}
\]
in $\azu$, with $i,j \in I$ and $B$ arbitrary, then the corresponding diagram
\[
\begin{tikzcd}[row sep=small]
& \trep_s A_i \ar[rd,"{\psi_i}"] & \\
\trep_s B \ar[ru] \ar[rd] & & \mathcal{G} \\
& \trep_s A_j \ar[ru,"{\psi_j}"'] &
\end{tikzcd}
\]
commutes as well. In particular, take a diagram
\[
\begin{tikzcd}[row sep=small]
& A_i \ar[rd,"{g}"] & \\
A \ar[ru,"{\phi_i}"] \ar[rd,"{\phi_i}"'] & & B \\
& A_i \ar[ru,"{\mathrm{id}}"'] &
\end{tikzcd}
\]
with $g$ an automorphism of $A_i$ fixing $A$ (or more precisely, fixing $\phi_i(A)$). Then it follows that
\[
\psi_i(g \cdot y) = \psi_i(y) \quad\text{for every section }y\text{ of the sheaf }\trep_s A_i.
\]
We now claim there is a unique
\[
\psi : \trep_s A \longrightarrow \mathcal{G}
\]
extending $\psi_i$ for every $i \in I$. Take a morphism $\wis{spec}(C) \to \trep_s A$ for some commutative ring $E$, given by an embedding
\[
x : A \to \M_s(E).
\]
We can take a $J$-cover $\{ E \to E_m \}_{m \in M}$, and some $\hat{m} \in I$ for every $m \in M$, and a morphism $x_{\hat{m}}$, such that the diagram
\[
\begin{tikzcd}
A \ar[r,"{x|_{E_m}}"] \ar[d,"{\phi_{\hat{m}}}"] & \M_s(E_m) \\
A_{\hat{m}} \ar[ru,"{x_{\hat{m}}}"'] & 
\end{tikzcd}
\]
commutes. Now define $\psi(x|_{E_m}) = \psi_{\hat{m}}(x_{\hat{m}})$. The image $\psi_{\hat{m}}(x_{\hat{m}})$ does not depend on the choice of $x_{\hat{m}}$, because for another choice $x_{\hat{m}}'$, we can ($J$-locally) find an automorphism $g \in \pgl_s(E_m)$ such that 
\[
x_{\hat{m}}' = g \cdot x_{\hat{m}}
\]
(see Proposition \ref{prop:homogeneous}). This $g$ leaves $A$ invariant, so we get
\[
\psi_{\hat{m}}(x_{\hat{m}}') = \psi_{\hat{m}}(g\cdot x_{\hat{m}}) = \psi_{\hat{m}}(x_{\hat{m}}).
\]
So, locally, $\psi(x)$ is uniquely determined. In particular it agrees on intersections, so it is defined globally. This shows that $\mathcal{G}^\sharp$ is a sheaf.

\underline{$1\simeq \sharp \circ \flat$.} By Yoneda Lemma, there is a natural bijective correspondence between elements $s \in \mathcal{F}(A)$ and presheaf morphisms $\mathbf{y}\!A \to \mathcal{F}$, or equivalently $(\mathbf{y}\!A)_{\downarrow s} \to \mathcal{F}$. Applying $\flat$ gives a morphism $\trep_s A \to \mathcal{F}^\flat$. This procedure yields a natural transformation
\[
\eta : \mathcal{F} \to (\mathcal{F}^\flat)^\sharp.
\]
We claim that this is an isomorphism. On  a matrix algebra $\M_n(C)$ with $C$ a commutative ring and $n \mid s$, the natural transformation $\eta: \mathcal{F} \to (\mathcal{F}^\flat)^\sharp$ is given by
\begin{gather*}
\mathcal{F}(\M_n(C)) \stackrel{\eta}{\longrightarrow} \Hom_{\pgl_s}( \trep_s \M_n(C), \mathcal{F}^\flat) \\
y \quad\mapsto\qquad
(\phi_y : \rho \mapsto \rho^* y).
\end{gather*}
For the standard embedding $\rho_{n,s} : \M_n(C) \to \M_s(C)$, we get $\phi_y(\rho_{n,s}) = \rho^*_{n,s}(y)$. Note that the restriction map $\rho^*_{n,s}$ on $\mathcal{F}$ is injective, because $\mathcal{F}$ is a $J_s$-sheaf, and $n \mid s$. So we can recover $y$ from $\phi_y$, in other words $\eta$ is injective.

To show surjectivity, take $\phi : \trep_s \M_n(C) \longrightarrow \mathcal{F}^\flat$ an arbitrary $\pgl_s$-equivariant morphism. Set $y = \phi(\rho_{n,s}) \in \varinjlim_{n \mid k}\mathcal{F}(\M_k(C))$. Because $\phi$ is equivariant, $y$ is invariant under all $g \in \pgl_s(C)$ that leave $\rho_{n,s}$ invariant. So we can interpret $y$ as an element of $\mathcal{F}(\M_n(C))$. It remains to show that $\phi = \phi_y$. Take an arbitrary $\rho: \M_n(C) \to \M_s(E)$ with $E$ some finitely generated commutative ring. Let $f: C \to E$ be the restriction of $\rho$ to $C$. Then $\rho_{n,s} \otimes f: \M_n(C) \to \M_s(E)$ is a center-preserving morphism with the same domain and codomain as $\rho$. Now by Proposition \ref{prop:homogeneous} we can find a Zariski covering $\{h_i: E \to E_i\}$ and automorphisms $g_i \in \pgl_s(E_i)$ such that
\[
\rho|_{E_i} = g_i \cdot (\rho_{n,s} \otimes f_i)
\]
with $f_i = h_i \circ f$. Now we can compute
\begin{gather*}
\phi(\rho|_{E_i}) 
= \phi(g_i \cdot (\rho_{n,s}\otimes f_i))
= g_i \cdot f_i^*\phi(\rho_{n,s}) \\
= g_i \cdot f_i^*y
= (g_i \cdot (\rho_{n,s}))^* y \\
= (\rho|_{E_i})^* y
= \phi_y(\rho|_{E_i}).
\end{gather*}
This shows that $\phi = \phi_y$. So $\eta$ is surjective.

The above shows that $\eta$ is an isomorphism on matrix algebras. From this it easily follows that it is an isomorphism on all Azumaya algebras, because $J_s$ is assumed to be trivializing.

\underline{$\flat \circ \sharp \simeq 1$.} For $\mathcal{G}$ a $\pgl_s$-sheaf, consider the natural transformation
\[
\varepsilon: (\mathcal{G}^\sharp)^\flat \to \mathcal{G}
\]
given by $\varepsilon(\phi) = \phi(\rho_{n,s})$, for
\[
\phi \in (\mathcal{G}^\sharp)^\flat(C) 
= \varinjlim_{n \mid s} \Hom_{\pgl_s}( \trep_s \M_n(C), \mathcal{G}).
\]
It is easy to see that $\phi(\rho_{n,s})$ does not depend on $n$, so $\varepsilon$ is well-defined. For $y \in \mathcal{G}(C)$, we have to show that there is a unique $\phi_y$ such that $y = \varepsilon(\phi_y) = \phi_y(\rho_{n,s})$. Here $n \in \NN_+$ is a natural number such that
\[
g \cdot y|_E = y|_E
\]
for all morphisms $C \to E$ and all elements $g \in \pgl_s(E)$ such that $g \sim_n 1$ (see Definition \ref{def:continuous}). We define $\phi_y$ as follows. Let $\rho: \M_n(C) \to \M_s(E)$ be a center-preserving morphism with $E$ some finitely generated commutative ring, and let $f: C \to E$ be the restriction of $\rho$ to $C$. Then by Proposition \ref{prop:homogeneous}, we can take a Zariski cover $\{ h_i: E \to E_i \}_{i \in I}$ and automorphisms $g_i \in \pgl_s(E_i)$ such that 
\[
\rho|_{E_i} = g_i \cdot (\rho_{n,s} \otimes f_i)
\]
with $f_i = h_i \circ f$. We partially define $\phi_y$ as
\[
\phi_y(\rho|_{E_i}) = g_i \cdot f_i^*y.
\]
This does not depend on the choice of $g_i$. Indeed, if $g_i'$ is another choice, then $g_i' g_i^{-1}\sim_n 1$ by definition, so $g_i \cdot f_i^*y = g_i' \cdot f_i^* y$. In particular, $\phi_y(\rho|_{E_i}) = \phi_y(\rho|_{E_j})$ agree on the intersection, for $i \neq j$. So we can define $\phi_y$ globally. Moreover, because $\phi_y$ has to be equivariant, this is the unique possibility.
\end{proof}

\section*{Acknowledgments}

I would like to thank Lieven Le Bruyn for the interesting discussions and for his comments on draft versions of this paper.

\bibliography{azumaya.bib}{}
\bibliographystyle{amsplain}
\end{document}